\documentclass[12pt]{amsart}
\usepackage{geometry}                
\usepackage{graphicx}
\usepackage{amssymb}
\usepackage{epstopdf}
\usepackage{ulem}

\numberwithin{equation}{section}
\usepackage[colorlinks=true]{hyperref}
\usepackage{enumerate}
\hypersetup{urlcolor=blue, citecolor=blue}

\newtheorem{theorem}{Theorem}[section]
\newtheorem{corollary}{Corollary}[section]

\newtheorem{lemma}[theorem]{Lemma}

\newtheorem{example}{Example}

\theoremstyle{definition}
\newtheorem{definition}[theorem]{Definition}
\newtheorem{remark}{Remark}

\newcommand{\R}{\mathbb R}

\newcommand{\h}{\mathcal{H}^{N-1}}

\def\per{\mathrm {Per}}

\def\dfrac{\displaystyle\frac}
\def\dint{\displaystyle\int}

\newcommand{\hs}{d\mathcal{H}^{n-1}}

\newcommand{\Ui}{\partial U_t^{int}}
\newcommand{\Ue}{\partial U_t^{ext}}

\newcommand{\Hm}{{\mathcal H}^{n-1}}

\title[ Talenti comparison]{ A Talenti comparison result for solutions to elliptic problems with Robin boundary conditions }

\author[ A. Alvino, C. Nitsch, C. Trombetti]{A. Alvino, C. Nitsch, C. Trombetti}
\date{}                                           
\address{\vskip1cm\noindent  \hfill\break\vskip-.2cm
\noindent  \hfill\break\vskip-.2cm
\noindent Dipartimento di Matematica e Applicazioni ``R. Caccioppoli'', Universit\`{a}
degli Studi di Napoli ``Federico II'', Complesso Universitario Monte S. Angelo, via Cintia
- 80126 Napoli, Italy. \hfill\break\vskip-.2cm
\noindent e-mail: \tt 
cristina@unina.it}


\begin{document}

\maketitle

\begin{abstract}
Comparison results of Talenti type for Elliptic Problems with Dirichlet boundary conditions have been widely investigated in the last decades. In this paper, we deal with Robin boundary conditions. Surprisingly, contrary to the Dirichlet case, Robin boundary conditions make the comparison sensitive to the dimension, and while the planar case seems to be completely settled, in higher dimensions some open problems are yet unsolved.
\end{abstract} 

\section{Introduction }

Let  $\beta $ be a positive parameter, and let $\Omega $ be an open, bounded set of $ \R^N$, $N\ge 2$ with Lipschitz boundary. For a given nonnegative ( not identically zero)  $f\in L^2(\Omega)$ we consider the following  problem
\begin{equation}\label{problem}
\left\{
\begin{array}{ll}
-\Delta u= f  & \mbox{in $\Omega$}\\\\
\dfrac{\partial u}{\partial \nu} +\beta \,u =0 & \mbox{on $\partial\Omega$,}
\end{array}
\right.
\end{equation}
where  $\nu$, denotes the outer unit normal to $\partial \Omega$. 

A function  $u \in H^1(\Omega)$ is a weak solution to \eqref{problem} if

\begin{equation}
\label{weaksol}
\int_{\Omega} \nabla u \nabla \phi \,dx + \beta\int_{\partial \Omega}u \phi \, \hs= \int_{\Omega} f \phi \,dx \quad \forall \phi \in H^1(\Omega).
\end{equation}

We will establish a comparison principle with the solution to the following problem
\begin{equation}\label{problem_sharp}
\left\{
\begin{array}{ll}
-\Delta v= f^\sharp & \mbox{in $\Omega^\sharp$}\\\\
\dfrac{\partial v}{\partial \nu} +\beta \,v =0 & \mbox{on $\partial\Omega^\sharp$.}
\end{array}
\right.
\end{equation}
where $\Omega^\sharp$ denotes the ball, centered at the origin, with the same Lebesgue measure as $\Omega$ and $f^\sharp$ is the decreasing Schwarz rearrangement of $f$.
Our main theorems are
\begin{theorem}\label{th_main_f}
Let $u$ and $v$  be the solution to Problem \eqref{problem} and to Problem  \eqref{problem_sharp}, respectively. Then we have
$$\|u\|_{L^{p,1}(\Omega)}\le \|v\|_{L^{p,1}(\Omega^\sharp)}\quad \mbox{for all } 0< p\le\frac{N}{2N-2},$$
$$\|u\|_{L^{2p,2}(\Omega)}\le \|v\|_{L^{2p,2}(\Omega^\sharp)} \quad \mbox{for all } 0< p\le\frac{N}{3N-4}.$$
\end{theorem}

\begin{theorem}\label{th_main_1}
Let assume $f\equiv 1$ and let $u$ and $v$ be the solutions to Problem \eqref{problem} and to Problem \eqref{problem_sharp}, respectively. Then, when $N=2$, we have
$$u^\sharp(x)\le v(x) \quad x \in \>\mbox{in $\Omega^\sharp$}.$$
While, when $N\ge 3$, we have
$$\|u\|_{L^{p,1}(\Omega)}\le \|v\|_{L^{p,1}(\Omega^\sharp)}\quad \mbox{and}\quad \|u\|_{L^{2p,2}(\Omega)}\le \|v\|_{L^{2p,2}(\Omega^\sharp)} \quad \mbox{for all } \>0< p\le\frac{N}{N-2}.$$
\end{theorem}

We remind that, for $0<p<\infty$ and $0<q\le\infty$, the Lorentz space $L^{p,q}(\Omega)$ consists of all measurable functions $g$ in $\Omega$ such that it is finite the quantity 
\[
\|g\|_{L^{p,q}(\Omega)}=
\left\{
\begin{array}{ll}
\displaystyle p^\frac1q\left( \int_0^\infty t^q\left| \left\{ x\in\Omega:|g(x)|> t \right\}\right|^\frac{q}{p}\frac{dt}{t} \right)^\frac1q & 0<q<\infty,\\\\
\displaystyle \sup_{t>0}\left(t^p\left| \left\{ x\in\Omega:|g(x)|> t \right\} \right|\right) & q=\infty.
\end{array}
\right.
\]
For $p=q$ (see \cite{Hu}), Lorentz spaces coincide with $L^p$ spaces (since $\|g\|_{L^{p,p}(\Omega)}=\|g\|_{L^{p}(\Omega)}$).
Therefore, in the hypothesis of Theorem \ref{th_main_f} when $N=2$ we have $\|u\|_{L^{1}(\Omega)}\le\|v\|_{L^{1}(\Omega^\sharp)}$ and $\|u\|_{L^{2}(\Omega)}\le\|v\|_{L^{2}(\Omega^\sharp)}$.
Moreover in the hypothesis of Theorem \ref{th_main_1} when $N\ge 3$ we have $\|u\|_{L^{1}(\Omega)}\le\|v\|_{L^{1}(\Omega^\sharp)}$ and $\|u\|_{L^{2}(\Omega)}\le\|v\|_{L^{2}(\Omega^\sharp)}$.\\

Comparison results \`a la Talenti have been widely studied in the last decades, after in his seminal paper \cite{Ta} Talenti proved that, if $u$ is the solution to
\begin{equation*}
\left\{
\begin{array}{ll}
-\Delta u= f  & \mbox{in $\Omega$}\\\\
u =0 & \mbox{on $\partial\Omega$,}
\end{array}
\right.
\end{equation*}
and $v$ is the solution to
\begin{equation}
\left\{
\begin{array}{ll}
-\Delta v= f^\sharp & \mbox{in $\Omega^\sharp$}\\\\
v =0 & \mbox{on $\partial\Omega^\sharp$,}
\end{array}
\right.
\end{equation}
then $u^\sharp(x)\le v(x)$ for all $x$ in $\Omega^\sharp$.

It is impossible to make a comprehensive list of all the results developed in the wake of this fundamental achievement. Generalization to semilinear and nonlinear elliptic equations are, for instance, in \cite{ALT, Ta2}, anisotropic elliptic operators are considered for instance in \cite{AFLT}, while parabolic equation are handled for instance in \cite{ALT}. Higher order operators have been investigated for instance in \cite{AB,Ta3} and two textbooks which provide survey on Talenti's technique and collect as well many other references are \cite{Ka, K1}.
However, to our knowledge, in literature there are no comparison results related to Talenti techniques, concerning Robin boundary conditions. We mention however that, when $f=1$, it has been proved in \cite{BG15} with a completely different argument that, if $u$ and $v$ are solutions to Problem \eqref{problem} and to Problem \eqref{problem_sharp}, respectively, then $\|u\|_{L^{1}(\Omega)}\le\|v\|_{L^{1}(\Omega^\sharp)}$.

The paper is organized as follows. In the next Section we introduce some basic notations, we recall the notion of decreasing rearrangements as well as some of their basic properties. Then in Section \ref{sec_proof} we provide a detailed proof of Theorems \ref{th_main_f} -- \ref{th_main_1}. In Section \ref{sec_BD} we show an alternative proof of the Bossel-Daners inequality for planar domains. Finally, in the last Section we provide some examples, we discuss the optimality of our results, and we provide a list of open problems.

\


\section{Notation and Preliminaries}
The solution $u \in H^1(\Omega)$ to \eqref{problem} is the unique minimizer of 

\begin{equation}
\label{minimizer}
\min_{w\in H^1(\Omega)} {\displaystyle \frac12\int_\Omega |\nabla w|^2 \, dx + \frac12\beta \displaystyle\int_{\partial \Omega} w^2 \, d \Hm }-{\displaystyle\int_\Omega  f w \, dx}.
\end{equation}


For $\displaystyle t\ge 0$ we denote by 

\[
U_t=\{x \in \Omega: u(x)>t\}, \quad \Ui = \partial U_t \cap \Omega, \quad \Ue= \partial U_t \cap \partial\Omega,
\]

and by 
\[
\mu(t) = |U_t|, \quad P_u(t) = \per(U_t).
\]
the Lebesgue measure of $U_t$ and its perimeter in $\R^N$, respectively.
Moreover, $\Omega^\sharp$ denotes the ball, centered at the origin, with the same measure as $\Omega$ and $v$ denotes the unique, radial and decreasing along the radius, solution to Problem\eqref{problem_sharp}.


Then, using the same notation as above, for $t\ge 0$ we set

\[
V_t=\{x \in \Omega^\sharp: v(x)>t\}, \quad
\phi(t) = |V_t|, \quad  \quad P_v(t) = \per(V_t).
\]
Since $v$ is radial, positive and decreasing along the radius then, for $0\le t\le\min_{\Omega^\sharp} v$, $V_t$ coincides with $\Omega^\sharp$,  while, for $\min_{\Omega^\sharp} v <t<\max_{\Omega^\sharp} v$, $V_t$ is a ball concentric to $\Omega^\sharp$ and strictly contained in it. 

In what follows we denote by $\omega_N$ the measure of the unit ball in $\R^N$.

\begin{definition}
\label{rarrangement}
Let $h: x \in \Omega \rightarrow [0,+\infty[$ be a measurable function, then the decreasing rearrangement $h^*$ of $h$ is defined as follows:
\[
h^*(s) = \inf\{t  \ge 0 : |\{x\in\Omega:| h(x)|>t\}| < s\} \quad s \in [0,\Omega].
\]
while the Scwartz rearrangement of $h$ is defined as follows
\[
h^\sharp(x) = h^*(\omega_N |x|^N) \quad x \in \Omega^\sharp.
\]
\end{definition}

It is easily checked that  $h$, $h^*$  and $h^\sharp$ a are equi-distributed, i.e. 
$$
|\{x\in\Omega: |h(x)|>t\}| = |\{s\in (0,|\Omega|: h^*(s)>t\}|  = |\{x\in\Omega^\sharp: h^\sharp(x)>t\}|\quad t\ge 0
$$
and then
if $h\in L^p(\Omega)$, $ 1 \le p \le \infty$, then $h^*  \in L^p(0,|\Omega|)$,  $h^\sharp \in L^p(\Omega^\sharp)$, 
and $$||h||_{L^p(\Omega)}=||h^*||_{L^p(0,|\Omega|)}=||h^\sharp||_{L^p(\Omega^\sharp)} .$$

%

Moreover, the following inequality, known as Hardy-Littlewood inequality, holds true
\begin{equation}  \label{hl}
 \int_\Omega |h(x)g(x)|dx \le
\int_0^{|\Omega|}h^*(s)g^*(s)ds.
\end{equation}
In the applications of the theory of rearrangements to the study of partial
differential equations one often has to evaluate the integral of a nonnegative function $
f\in L^p(\Omega)$, $1 \le p \le +\infty$, on the level sets of a measurable
function $u$. By \eqref{hl} we get
\begin{equation}  \label{fu}
\int_{|u(x)|>t} f(x) dx \le \int_0^{\mu_u(t)}f^\ast(s)ds
\end{equation}
Clearly, if we take $f=u \ge 0$ in \eqref{fu} we have 
\begin{equation}\label{uu}
\int_{u(x)>t} u(x) dx = \int_0^{\mu_u(t)}u^\ast(s)ds.
\end{equation}

\section{Proof of Theorem \ref{th_main_f} and Theorem \ref{th_main_1}}\label{sec_proof}
As a premise to the proof of our main results we remind the following lemma.

\begin{lemma}[Gronwall]\label{lem_Gronwall}
Let $\xi(\tau)$ be a continuously differentiable function 
satisfying, for some non negative constant $C$, the following differential inequality  
\[
\tau \xi'(\tau) \le \xi(\tau)+C \qquad \mbox{for all $\tau\ge\tau_0>0$.}
\] 
Then we have
\begin{enumerate}[i)]
\item \[ \xi(\tau) \le \tau\frac{\xi(\tau_0)+C}{\tau_0}-C \qquad \mbox{for all $\tau\ge\tau_0$},\]
\item\label{ineq_Gronwall} \[ \xi'(\tau) \le \frac{\xi(\tau_0)+C}{\tau_0} \qquad \mbox{for all $\tau\ge\tau_0$}.\]
\end{enumerate}
\end{lemma}
The main ingredient for a comparison result is the following lemma.
\begin{lemma}
Let $u$ and $v$ be the solution to \eqref{problem} and \eqref{problem_sharp}, respectively.
For a.e. $t>0$ we have
\begin{equation}\label{eq_fundamental}\gamma_N \phi(t)^\frac{2N-2}{N}=\left(-\phi'(t)+\dfrac{1}{\beta} \dint_{\partial V_t\cap \partial\Omega^\sharp} \dfrac{1}{v(x)} \>d\h(x) \right)\int_0^{\phi(t)}f^*(s)ds,
\end{equation}
while for almost all $t>0$ it holds
\begin{equation}\label{ineq_fundamental}\gamma_N \mu(t)^\frac{2N-2}{N}\le\left(-\mu'(t)+\dfrac{1}{\beta} \dint_{\Ue} \dfrac{1}{u(x)} \>d\h(x) \right)\int_0^{\mu(t)}f^*(s)ds.
\end{equation}
Here $\gamma_N=N^2\omega_N^{-\frac{2}{N}}$.
\end{lemma}

\begin{proof}

Let $t>0$ and  $h>0$, and let us choose the following test function in \eqref{weaksol} 
\begin{equation}
\varphi_h(x)= \left\{
\begin{array}{ll}
0 & \mbox{if $0<u<t$}\\\\
h & \mbox{if $u> t+h$} \\\\
u-t  &\mbox{if $t<u<t+h$}.
\end{array}
\right.
\end{equation}

Then,

\begin{equation}
\begin{array}{ll}
 \dint_{U_t \setminus  U_{t+h}} |\nabla u|^2 \, dx + \beta h \dint_{\partial U_{t+h}^{ext}} u  \, d\h(x)+ & \beta \dint_{\partial U_{t}^{ext} \setminus \partial U_{t+h}^{ext}} u (u-t) \, d\h(x)=  \\\\
& \dint_{U_t \setminus U_{t+h}} f (u-t) \, dx + h \dint_{U_{t+h} } f \, dx
\end{array}
\end{equation}
dividing by $h$ and letting $h$ go to, using coarea formula we have that for a.e. $t>0$

\begin{equation}
\int_{\partial U_t} g(x) \,d\h(x) = \int_{ \Ui} |\nabla u| \, d\h(x) + \beta \int_{ \Ue} u d\h(x) = \int_{U_t} f\, dx
\end{equation}

where
\[
g(x) = \left\{ \begin{array}{ll}
|\nabla u| & \mbox{if $x \in \Ui$}\\\\
\beta u & \mbox{if $x \in \Ue$}\\\\
\end{array}
\right.
\]
for a.e.$t>0$ we have
\begin{equation}
\label{estimate1}
\begin{array}{ll}
 P_u^2(t) & \leq  \left(\dint_{\partial U_t} g(x) d\h(x)  \right) \left(\dint_{\partial U_t}  g(x)^{-1} d\h(x)  \right)= \\ \\
&\left( \dint_{\partial U_t} g(x) d\h(x)  \right)\left(\dint_{\Ui } |\nabla u |^{-1} d\h(x)  + \dint_{\Ue}   (\beta u )^{-1} d\h(x) \right) \leq \\ \\
&\displaystyle\int_0^{\mu(t)}f^*(s)ds \left( -\mu'(t) + \dfrac{1}{\beta} \dint_{\Ue} \dfrac{1}{u} \>d\h(x)  \right) \quad t \in [0,\max_{\Omega} u).
\end{array}
\end{equation}
Then the isoperimetric inequality gives \eqref{ineq_fundamental}.
If $v$ solves Problem \eqref{problem}, all the previous inequalities hold as equalities hence \eqref{eq_fundamental} follows.
\end{proof}

\begin{remark}
We observe that solutions $u$ and $v$ to Problem \eqref{problem} and Problem \eqref{problem_sharp}, always achieve their minima on the boundary of $\Omega$ and $\Omega^\sharp$ respectively. From now on we denote by $$v_m=\min_{\Omega^\sharp} v, \>u_m=\min_{\Omega} u.$$ The following inequality holds true
\begin{equation}\label{ineq_bordo}
u_m\le v_m.
\end{equation}
In fact, using the equations and the boundary conditions in \eqref{problem} and \eqref{problem_sharp}, 
\begin{equation}
\begin{array}{rl}
v_m \per(\Omega^\sharp)=&\dint_{\partial\Omega^\sharp}v(x)\>d\h(x)\\\\
=&\dint_{\partial\Omega}u(x)\>d\h(x)\ge u_m \per(\Omega)\ge u_m \per(\Omega^\sharp).
\end{array}
\end{equation}
An important consequence of \eqref{ineq_bordo}, is that 
\begin{equation}\label{ineq_iniziale}
\mu(t)\le \phi(t)=|\Omega|\qquad \mbox{for all $0\le t \le v_m$.}
\end{equation}
With strict inequality for some $0\le t \le v_m$ unless $\Omega$ is a ball.
\end{remark}

A fundamental lemma which allows us to estimate the boundary integral on the right hand side on \eqref{estimate1} is the following.
\begin{lemma}\label{lem_boundary}
For all $t\ge v_m$ we have
\begin{equation}\label{eq_boundary}
\int_0^t \tau\left(\dint_{\partial V_\tau\cap \partial\Omega^\sharp} \dfrac{1}{v(x)} \> d\h(x) \right)\,d\tau = \frac{\dint_0^{|\Omega|}f^*(s)ds}{2\beta},
\end{equation}
while
\begin{equation}\label{ineq_boundary}
\int_0^t \tau\left(\dint_{\partial U_\tau^{ext}} \dfrac{1}{u(x)} \>d\h(x) \right)\,d\tau \le \frac{\dint_0^{|\Omega|}f^*(s)ds}{2\beta}.
\end{equation}
\end{lemma}
\begin{proof}
By Fubini's theorem and using \eqref{problem} we have
\begin{equation*}
\begin{array}{ll}
\dint_0^\infty \tau\left(\dint_{\partial U_\tau^{ext}} \dfrac{1}{u(x)} \>d\h(x) \right)\,d\tau &= \dint_{\partial \Omega} \left(\int_0^{u(x)}\dfrac{\tau}{u(x)} \>d\tau\, \right)d\h(x)\\
&= \dint_{\partial \Omega} \dfrac{u(x)}2 \>d\h(x)=\frac{\dint_0^{|\Omega|}f^*(s)ds}{2\beta}.
\end{array}
\end{equation*}
Analogously,
\begin{equation*}
\dint_0^\infty \tau\dint_{\partial V_\tau\cap\partial\Omega^\sharp} \dfrac{1}{v(x)} \>d\h(x) \,d\tau
=\frac{\dint_0^{|\Omega|}f^*(s)ds}{2\beta}.
\end{equation*}
Therefore, one trivial inequality for $t\ge 0$ is
\[
\dint_0^t \tau\dint_{\partial U_\tau^{ext}} \dfrac{1}{u(x)} \>d\h(x) \,d\tau \le \dint_0^\infty \tau\dint_{\partial U_\tau^{ext}} \dfrac{1}{u(x)} \>d\h(x) \,d\tau,
\]
while we observe that for $t\ge v_m=\min_{\partial\Omega^\sharp} v$ then $\partial V_t \cap \partial \Omega^\sharp = \emptyset$
\[
\dint_0^t \tau\dint_{\partial V_\tau\cap\partial\Omega^\sharp} \dfrac{1}{v(x)} \>d\h(x) \,d\tau=\dint_0^\infty \tau\dint_{\partial V_\tau\cap\partial\Omega^\sharp} \dfrac{1}{v(x)} \>d\h(x) \,d\tau.
\]
\end{proof}

\begin{proof}[Proof of Theorem \ref{th_main_f}]
Let $0 <p\le\frac{N}{2N-2}$. We start by multiplying \eqref{ineq_fundamental} by $t\mu(t)^{\frac{1}{p}-\frac{2N-2}{N}}$, then we integrate from $0$ to some $\tau\ge v_m$, and we use Lemma \ref{lem_boundary}, deducing
\begin{equation*}
\begin{array}{rl}
\dint_0^\tau \gamma_N t\mu(t)^\frac1p \,dt \le& \dint_0^\tau-\mu'(t)t\mu(t)^{\delta}\left(\int_0^{\mu(t)}f^*(s)ds\right)\,dt\\\\
&+\dfrac{|\Omega|^{\delta}}{2\beta^2}{\left(\dint_0^{|\Omega|}f^*(s)ds\right)^2}.
\end{array}
\end{equation*}
Here we have used \eqref{ineq_boundary} and we have set $\displaystyle\delta=\frac{1}{p}-\frac{2N-2}{N}$.
Even if $\mu(t)$ is not necessarily absolutely continuous, using the fact that $\mu(t)$ is a monotone non increasing function, we can conclude that
\begin{equation}\label{ineq_basic}
\begin{array}{rl}
\dint_0^\tau \gamma_N t\mu(t)^\frac1p dt\le& \dint_0^\tau-t\mu(t)^{\delta}\left(\int_0^{\mu(t)}f^*(s)ds\right)\,d\mu(t)\\\\
&+\dfrac{|\Omega|^{\delta}}{2\beta^2}{\left(\dint_0^{|\Omega|}f^*(s)ds\right)^2}.
\end{array}
\end{equation}
This allows us to integrate by parts both sides of the last inequality, but first we set $F(\ell)=\dint_0^\ell w^\delta\left(\int_0^w f^*(s)ds\right)\,dw$. Hence we have for $\tau \ge v_m$
\begin{equation*}
\begin{array}{rl}
\tau F(\mu(\tau))+\tau\dint_0^\tau \gamma_N \mu(t)^\frac1p dt\le& \dint_0^\tau F(\mu(t))dt+\int_0^\tau\int_0^t \gamma_N\mu(t)^\frac{1}{p} dr\, dt\\\\ 
&\displaystyle+\frac{|\Omega|^{\delta}}{2\beta^2}{\left(\dint_0^{|\Omega|}f^*(s)ds\right)^2}.
\end{array}
\end{equation*}
Using the same notation of Lemma \ref{lem_Gronwall} we set $$\xi(\tau)=\int_0^\tau F(\mu(t))dt+\int_0^\tau\left(\int_0^t \gamma_N\mu(r)^\frac{1}{p} dr\right)\, dt,$$ $C=\frac{|\Omega|^{\delta}}{2\beta^2}{\left(\dint_0^{|\Omega|}f^*(s)ds\right)^2},$ and $\tau_0=v_m$. Thereafter we deduce from Lemma \ref{lem_Gronwall} that 
\begin{equation}\label{ineq_afterGu}
\begin{array}{ll}
F(\mu(\tau))&+\dint_0^\tau \gamma_N \mu(t)^\frac1p dt\le\dfrac{1}{v_m}\Bigg(\int_0^{v_m} F(\mu(t))dt\\\\
&+\dint_0^{v_m}\int_0^t \gamma_N\mu(r)^\frac{1}{p} dr\, dt 
+ \dfrac{|\Omega|^{\delta}}{2\beta^2}{\left(\dint_0^{|\Omega|}f^*(s)ds\right)^2}\,\Bigg).
\end{array}
\end{equation}
Similarly, arguing in the same way with \eqref{eq_fundamental}, we can deduce
\begin{equation}\label{ineq_afterGv}
\begin{array}{ll}
F(\phi(\tau))&+\dint_0^\tau \gamma_N \phi(t)^\frac1p dt= \dfrac{1}{v_m}\Bigg(\int_0^{v_m} F(\phi(t))dt\\\\
&+\dint_0^{v_m}\dint_0^t \gamma_N\phi(r)^\frac{1}{p} dr\, dt + \frac{|\Omega|^{\delta}}{2\beta^2}{\left(\dint_0^{|\Omega|}f^*(s)ds\right)^2}\,\Bigg)=\\\
& F(|\Omega|) + \dfrac{\gamma_N |\Omega|^\frac{1}{p} v_m}{2} +\frac{|\Omega|^{\delta}}{2\beta^2}{\left(\dint_0^{|\Omega|}f^*(s)ds\right)^2}.
\end{array}
\end{equation}
Taking into account \eqref{ineq_bordo}, it is possible a direct comparison between the righthand sides \eqref{ineq_afterGu} and \eqref{ineq_afterGv}, yielding
$$F(\phi(\tau))+\int_0^\tau \gamma_N \phi(t)^\frac1p\ge F(\mu(\tau))+\int_0^\tau \gamma_N \mu(t)^\frac1p.$$
Passing to the limit as $\tau\to \infty$ we get
$$\int_0^\infty \mu(t)^\frac1p dt\le\dint_0^\infty \phi(t)^\frac1p dt,$$
and hence 
$$\|u\|_{L^{p,1}(\Omega)}\le \|v\|_{L^{p,1}(\Omega^\sharp)}.$$

To conclude the proof we show that  $\mbox{for all } 0< p\le\frac{N}{3N-4}$
$$ \int_0^\infty t\mu(t)^\frac1p dt\le \int_0^\infty t\phi(t)^\frac1p dt.$$
To this aim, we consider the limit as $\tau\to\infty$ in \eqref{ineq_basic} and integrate by parts the first term on the righthand side, to obtain
$$\int_0^\infty \gamma_N t\mu(t)^\frac1p dt \le \int_0^\infty F(\mu(t))dt+\frac{|\Omega|^{\delta}}{2\beta^2}{\left(\dint_0^{|\Omega|}f^*(s)ds\right)^2}.$$
On the other hand $$\int_0^\infty \gamma_N t\phi(t)^\frac1p dt=\int_0^\infty F(\phi(t))dt+\frac{|\Omega|^{\delta}}{2\beta^2}{\left(\dint_0^{|\Omega|}f^*(s)ds\right)^2}.$$
Therefore it is enough to show that 
\begin{equation}\label{ineq_desired}
\int_0^\infty F(\mu(t))dt\le \int_0^\infty F(\phi(t))dt.
\end{equation}
This can be done for instance multiplying \eqref{ineq_fundamental} by $t F(\mu(t)) \mu(t)^{-\frac{2N-2}{N}}.$
Since $\mbox{for } 0< p\le\frac{N}{3N-4}$the function $F(\ell) \ell^{-\frac{2N-2}{N}}$ is non decreasing in $\ell$, an integration from $0$ to any $\tau\ge v_m$ yields
\begin{equation}\label{ineq_F}
\begin{array}{ll}
\dint_0^\tau \gamma_N tF(\mu(t))dt \le& \dint_0^\tau-t\mu^{-\frac{2N-2}{N}}F(\mu(t))\left(\dint_0^{\mu(t)}f^*(s)ds \right)\,d\mu(t)\\
&+F(|\Omega|)\dfrac{|\Omega|^{-\frac{2N-2}{N}}}{2\beta^2}{\left(\dint_0^{|\Omega|}f^*(s)ds\right)^2}.
\end{array}
\end{equation}
We now set $C=F(|\Omega|)\dfrac{|\Omega|^{-\frac{2N-2}{N}}}{2\beta}{\left(\dint_0^{|\Omega|}f^*(s)ds\right)^2}$ and $H(\ell)=\dint_0^\ell w^{-\frac{2N-2}{N}}F(w)\int_0^{w}f^*(s)ds\,dw,$ and after an integration by parts on both sides of \eqref{ineq_F} we have
$$\tau\int_0^\tau\gamma_N F(\mu(t))dt+\tau H(\mu(\tau))\le \int_0^\tau\int_0^r\gamma_N F(\mu(z))dz\,dr+\int_0^\tau H(\mu(t))dt+C.$$
As before we use Lemma \ref{lem_Gronwall} with $\xi(\tau)=\dint_0^\tau\int_0^r\gamma_N F(\mu(z))dz\,dr+\int_0^\tau H(\mu(t))dtdt$, and $\tau_0=v_m$. Thereafter we have
\begin{equation*}
\begin{array}{ll}
&\dint_0^\tau\gamma_N F(\mu(t))dt+H(\mu(\tau))\\\\
\le&\dfrac{1}{v_m}\left(\dint_0^{v_m}\dint_0^r\gamma_N F(\mu(z))dz\,dr+\int_0^{v_m} H(\mu(t))dt+C\right).
\end{array}
\end{equation*}
Since the previous inequality holds as an equality whenever $\mu$ is replaced by $\phi$ we easily deduce 
$$\int_0^\tau\gamma_N F(\mu(t))dt+ H(\mu(\tau))\le\int_0^\tau\gamma_N F(\phi(t))dt+ H(\phi(\tau))$$ and for $\tau \to \infty$ the desired inequality \eqref{ineq_desired} which concludes the proof.

\end{proof}

\begin{proof}[Proof of the Theorem \eqref{th_main_1}]
We start with the case $N=2$. We multiply by $t$\ref{ineq_fundamental}  and we  integrate from $0$ to $\tau$ choosing $\tau \ge v_m$. Taking into account that now $$\int_0^{\mu(t)}f^*(s)ds=\mu(t),$$ we have
 \begin{equation*}
2 \pi \tau^2  \leq \dint_0^\tau -\mu'(t) t \,dt+ \dfrac{|\Omega|}{2 \beta^2} \quad \mbox{for } \tau\ge v_m.
\end{equation*}
At the same time equality holds true whenever $\mu$ is replaced by $\phi$ and therefore
 \begin{equation*}
2 \pi \tau^2  =\dint_0^\tau -\phi'(t) t \,dt+ \dfrac{|\Omega|}{2 \beta^2} \quad \mbox{for } \tau\ge v_m.
\end{equation*}
Then,
\begin{equation}\label{estimate6}
\dint_0^\tau  t \,(-d \mu(t) )\geq \dint_0^\tau  t \,(-d \phi(t) ) \quad \tau \ge v_m.
\end{equation}
Then and integration by parts gives
\begin{equation}
\label{estimate7}
\mu(\tau) \leq \phi(\tau) \quad \tau \ge v_m.
\end{equation}
Since \eqref{ineq_bordo} is in force, inequality \eqref{estimate7} follows for $ t\ge 0$ and the claim is proved.

Now we consider $N\ge3$. Equation \eqref{ineq_fundamental} reads as follows
\begin{equation}\label{ineq_fund1}
\gamma_N \mu(t)^\frac{N-2}{N}\le\left(-\mu'(t)+\dfrac{1}{\beta} \dint_{\Ue} \dfrac{1}{u(x)} \>d\h(x) \right).
\end{equation}
Let $q\le\frac{N}{N-2}$. We multiply \eqref{ineq_fundamental} by $t\mu(t)^{\frac{1}{q}-\frac{N-2}{N}}$ and integrate from $0$ to some $\tau\ge v_m$. Then we use Lemma \ref{lem_boundary} to deduce
\begin{equation*}
\int_0^\tau \gamma_N t\mu(t)^\frac1q \,dt \le \int_0^\tau-\mu'(t)t\mu(t)^{\eta}dt+\frac{|\Omega|^{\eta+1}}{2\beta^2}.
\end{equation*}

Here we have used \eqref{ineq_boundary} and we have set $\displaystyle\eta=\frac{1}{q}-\frac{N-2}{N}$.
As before, using the fact that $\mu(t)$ is a monotone non increasing function, we can write
\begin{equation}\label{ineq_basic1}
\int_0^\tau \gamma_N t\mu(t)^\frac1q dt\le \int_0^\tau-t\mu(t)^{\eta}\,d\mu(t)+\frac{|\Omega|^{\eta+1}}{2\beta^2}.
\end{equation}
We set $G(\ell)=\dint_0^\ell w^\eta= \frac{\ell^{\eta+1}}{\eta+1} $ and after an integration by parts the last inequality reads
\begin{equation*}
\tau G(\mu(\tau))+\tau\int_0^\tau \gamma_N \mu(t)^\frac1q dt\le \int_0^\tau G(\mu(t))dt +\int_0^\tau\int_0^t \gamma_N\mu(t)^\frac{1}{q} dr\, dt+\frac{|\Omega|^{\eta+1}}{2\beta^2}.
\end{equation*}
We can then use Lemma \ref{lem_Gronwall} with $$\xi(\tau)=\int_0^\tau G(\mu(t))dt+\int_0^\tau\int_0^t \gamma_N\mu(t)^\frac{1}{q} dr\, dt,$$ $C=\frac{|\Omega|^{\eta}}{2\beta}{\left(\dint_0^{|\Omega|}f^*(s)ds\right)}$, and $\tau_0=v_m$, to deduce from Lemma \ref{lem_Gronwall} \ref{ineq_Gronwall}) that
\begin{equation}\label{ineq_1Gu}
\begin{array}{ll}
\displaystyle G(\mu(\tau))&+\displaystyle\int_0^\tau \gamma_N \mu(t)^\frac1q dt\le \frac{1}{v_m}\Bigg(\int_0^{v_m} G(\mu(t))dt\\\\
& \displaystyle+\int_0^{v_m}\int_0^t \gamma_N\mu(r)^\frac{1}{q} dr\, dt + \frac{|\Omega|^{\eta+1}}{2\beta^2}\Bigg).
\end{array}
\end{equation}
Similarly, arguing in the same way with \eqref{eq_fundamental}, we can deduce
\begin{equation}\label{ineq_1Gv}
\begin{array}{ll}
\displaystyle G(\phi(\tau))&+\displaystyle\int_0^\tau \gamma_N \phi(t)^\frac1q dt= \frac{1}{v_m}\Bigg(\int_0^{v_m} G(\phi(t))dt\\\\
&\displaystyle+\int_0^{v_m}\int_0^t \gamma_N\phi(r)^\frac{1}{q} dr\, dt + \frac{|\Omega|^{\eta+1}}{2\beta^2}\Bigg).
\end{array}
\end{equation}
Taking into account \eqref{ineq_bordo}--\eqref{ineq_iniziale}, it is possible a direct comparison between the righthand sides \eqref{ineq_1Gu} and \eqref{ineq_1Gv}, yielding
$$G(\mu(\tau))+\int_0^\tau \gamma_N \mu(t)^\frac1q\le G(\phi(\tau))+\int_0^\tau \gamma_N \phi(t)^\frac1q.$$
Passing to the limit as $t\to \infty$ we get
$$\int_0^\infty \mu(t)^\frac1q dt\le\int_0^\infty \phi(t)^\frac1q dt$$
and hence 
$$\|u\|_{L^{q,1}(\Omega)}\le \|v\|_{L^{q,1}(\Omega^\sharp)}.$$

To conclude the proof we have to show that 
$$ \int_0^\infty t\mu(t)^\frac1q dt\le \int_0^\infty t\phi(t)^\frac1q dt.$$
To this aim, we consider the limit as $\tau\to\infty$ in \eqref{ineq_basic1} and integrating by parts the first term on the righthand side, to obtain
$$\int_0^\infty \gamma_N t\mu(t)^\frac1q \le \int_0^\infty G(\mu(t))dt+\frac{|\Omega|^{\eta+1}}{2\beta^2}.$$
On the other hand $$\int_0^\infty \gamma_N t\phi(t)^\frac1q =\int_0^\infty G(\phi(t))dt+\frac{|\Omega|^{\eta+1}}{2\beta^2}.$$
Therefore it is enough to show that 
\begin{equation}\label{ineq_desired1}
\int_0^\infty G(\mu(t))dt\le \int_0^\infty G(\phi(t))dt.
\end{equation}
This can be done for instance multiplying \eqref{ineq_fund1} by $t G(\mu(t)) \mu(t)^{-\frac{N-2}{N}}.$
Since the function $G(\ell) \ell^{-\frac{N-2}{N}} = \ell^{\eta+ 2/N}$ is non decreasing in $\ell$, an integration from $0$ to any $\tau\ge v_m$ yields
\begin{equation}\label{ineq_G}
\begin{array}{ll}
\displaystyle\int_0^\tau \gamma_N tG(\mu(t))dt \le&\displaystyle \int_0^\tau-t\mu^{-\frac{N-2}{N}}G(\mu(t))\,d\mu(t)\\\\
&\displaystyle+G(|\Omega|)\frac{|\Omega|^{\frac{2}{N}}}{2\beta^2}.
\end{array}
\end{equation}
We now set $C=G(|\Omega|)\frac{|\Omega|^{\frac{2}{N}}}{2\beta}$, $J(\ell)=\dint_0^\ell w^{-\frac{N-2}{N}}G(w)\,dw$ and after an integration by parts on both sides of \eqref{ineq_G} we have
$$\tau\int_0^\tau\gamma_N G(\mu(t))dt+\tau J(\mu(\tau))\le \int_0^\tau\int_0^r\gamma_N G(\mu(z))dz\,dr+\int_0^\tau J(\mu(t))dt+C.$$
As before, we use Lemma \ref{lem_Gronwall} with $$\displaystyle \xi(\tau)=\int_0^\tau\int_0^r\gamma_N G(\mu(z))dz\,dr+\int_0^\tau J(\mu(t))dtdt,$$ and $\tau_0=v_m$. Thereafter we deduce that
$$\int_0^\tau\gamma_N G(\mu(t))dt+J(\mu(\tau))\le \frac{1}{v_m}\left(\int_0^{v_m}\int_0^r\gamma_N G(\mu(z))dz\,dr+\int_0^{v_m} J(\mu(t))dt+C\right).$$
Since the previous inequality hold as an equality whenever $\mu$ is replaced by $\phi$ we easily infer
$$\int_0^\tau\gamma_N G(\mu(t))dt+ J(\mu(\tau))\le\int_0^\tau\gamma_N G(\phi(t))dt+ J(\phi(\tau))$$ and for $\tau \to \infty$ the desired inequality \eqref{ineq_desired1} which concludes the proof.
\end{proof}

\section{The Bossel-Daners inequality (an alternative proof)}\label{sec_BD}

We conclude with a remark concerning the first Robin-Laplacian eiegenvalue defined by 
\begin{equation}
\label{classical}
\lambda_{1,\beta} (\Omega) = \mathop{\min_{w\in H^1(\Omega)}}_{ w\ne 0} \dfrac {\displaystyle \int_\Omega |\nabla w|^2 \, dx + \beta \displaystyle\int_{\partial \Omega} w^2 \, d \Hm }{\displaystyle\int_\Omega  w^2 \, dx}.
\end{equation}

It is known that a Faber-Krahn inequality (namely the so called Bossel-Daners inequality) for such eigenvalue holds true(\cite{Bo2,BD,BG10,BG15,D1})  that is

\begin{equation}
\label{FK}
\lambda_{1,\beta} (\Omega) \ge \lambda_{1,\beta} (\Omega^{\sharp}).
\end{equation}

Following an idea contained in \cite{K} we have the following

\begin{corollary}
Let $N=2$, then Theorem \ref{th_main_f} implies  \eqref{FK}.
\end{corollary}

\begin{proof}
Let $u$ be the first Robin eigenfunction associated to  $\lambda_{1,\beta} (\Omega)$, then $u$ solves 
\begin{equation}\label{EP}
\left\{
\begin{array}{ll}
-\Delta u= \lambda_{1,\beta} (\Omega) u  & \mbox{in $\Omega$}\\\\
\dfrac{\partial u}{\partial \nu} +\beta \,u =0 & \mbox{on $\partial\Omega$.}
\end{array}
\right.
\end{equation}
Denoting by $z$ the solution to 
\begin{equation}\label{EPS}
\left\{
\begin{array}{ll}
-\Delta z= \lambda_{1,\beta} (\Omega) u^\sharp  & \mbox{in $\Omega^\sharp$}\\\\
\dfrac{\partial z}{\partial \nu} +\beta \,z=0 & \mbox{on $\partial\Omega^\sharp$.}
\end{array}
\right.
\end{equation}
Theorem  \ref{th_main_f} gives
\begin{equation}
\int_{\Omega} u^2 = \int_{\Omega^\sharp} (u^\sharp)^2 \le \int_{\Omega^\sharp} z^2
\end{equation}
then, by Cauchy - Schwarz inequality

\begin{equation}
\label{CS}
\int_{\Omega^\sharp} u^\sharp \, z \le \int_{\Omega^\sharp} z^2
\end{equation}
Multiplying equation \eqref{EPS} by $z$ and integrating 

\[
\lambda_{1,\beta} (\Omega) = \dfrac {\displaystyle \int_{\Omega^\sharp} |\nabla z|^2 \, dx + \beta \displaystyle\int_{\partial {\Omega^\sharp}} z^2 \, d \Hm }{\displaystyle\int_{\Omega^\sharp}  u^\sharp z \, dx} \ge 
\dfrac {\displaystyle \int_{\Omega^\sharp} |\nabla z|^2 \, dx + \beta \displaystyle\int_{\partial {\Omega^\sharp}} z^2 \, d \Hm }{\displaystyle\int_{\Omega^\sharp} z^2 \, dx}  \ge \lambda_{1,\beta} ({\Omega^\sharp})
\]

\end{proof}

\section{Conclusions and open problems}

Contrary to the classical comparison principle \cite{Ta} for the Poisson equation with Dirichlet boundary conditions, our result in general establishes only comparison in Lorentz spaces. In the hypothesis of Theorem \ref{th_main_1} when $N=2$ we have $u^\sharp\le v$ in $\Omega^\sharp$, therefore the following question arise.\\

\noindent{\bf Open Problem 1} Is, in the hypothesis of Theorem \ref{th_main_1}, $u^\sharp\le v$ in $\Omega^\sharp$ even for $N\ge 3$?\\

We already observed in the Introduction that in the hypothesis of Theorem \ref{th_main_f} when $N=2$ we have $\|u\|_{L^{1}(\Omega)}\le\|v\|_{L^{1}(\Omega^\sharp)}$ and $\|u\|_{L^{2}(\Omega)}\le\|v\|_{L^{2}(\Omega^\sharp)}$ and one may ask whether $\|u\|_{L^{p}(\Omega)}\le\|v\|_{L^{p}(\Omega^\sharp)}$ for other values of $p$. We know for sure that for large value of $p$ the answer is No. Example \ref{counter1} serves as a counterexample when $p=\infty$ and $N=2$. While Example \ref{counter2} serves as a counterexample for $p=2$ and $N=3$. Nevertheless the following question is still unsolved.\\

\noindent{\bf Open Problem 2} Is, in the hypothesis of Theorem \ref{th_main_f}, $\|u\|_{L^{1}(\Omega)}\le\|v\|_{L^{1}(\Omega^\sharp)}$ even for $N\ge 3$?\\

\begin{example}\label{counter1}
Consider $\Omega\subset\R^2$ equal to the union of two disjoint disk $D_1$ and $D_r$ with radii $1$ and $r$, respectively. We choose $\beta=\frac12$ and we fix $f=1$ on $D_1$ and $f=0$ on $D_r$. Both $u$ and $v$ in Theorem \ref{th_main_f} can be explicitly computed.\\
We have $\|u\|_{L^{\infty}(\Omega)}- \|v\|_{L^{\infty}(\Omega^\sharp)}=\frac{r^2}{4}+o(r^2)$
 \end{example}
\begin{example}\label{counter2}
Consider $\Omega\subset\R^3$ equal to the union of two disjoint balls $B_1$ and $B_r$ with radii $1$ and $r$, respectively. We choose $\beta=\frac12$ and we fix $f=1$ on $B_1$ and $f=0$ on $B_r$. Both $u$ and $v$ in Theorem \ref{th_main_f} can be explicitly computed. \\
We have $\|u\|^2_{L^{2}(\Omega)}- \|v\|^2_{L^{2}(\Omega^\sharp)}=\frac{8}{135}r^3+o(r^3)$
\end{example}

\section*{Acknowledgement}

This work was supported by GNAMPA grant 2018 "Aspetti Geometrici delle EDP e Disuguaglianze Funzionali in forma ottimale". We are grateful to Dorin Bucur for many interesting discussions and for having inspired Examples \ref{counter1}--\ref{counter2}.


\begin{thebibliography}{99}  


\bibitem{AFLT} A. Alvino, V. Ferone, P. L. Lions \& G. Trombetti. Convex symmetrization and applicatons, Annales de l'I.H.P. 14
(1997), 275--293--65.



\bibitem{ALT} A. Alvino, P. L. Lions \& G. Trombetti. Comparison results for elliptic and parabolic equations via Schwarz symmetrization, Annales de l'I.H.P. 7
(1990), 37--65.

\bibitem{AB} M.S. Ashbaugh \& R.D. Benguria.  On Rayleigh's conjecture for the clamped plate and its generalization to three dimensions. Duke Math J. 78 (1995) 1--17.

\bibitem{BS} C. Bennett \& R.  Sharpley. Interpolation of operators. Pure and Applied Mathematics, 129. Academic Press, Inc., Boston, MA, 1988. 

\bibitem{Bo2} M.H. Bossel. Membranes \'elastiquement li\'ees: inhomog\`enes ou sur une surface: une nouvelle extension du th\'eor\`eme isop\'erim\'etrique de Rayleigh-Faber-Krahn. Z. Angew. Math.
Phys. 39(5) (1988), 733--742.

\bibitem{BD} D. Bucur \& D. Daners.  An alternative approach to the Faber-Krahn inequality for Robin problems.
Calc. Var. Partial Differential Equations 37 (2010),  75--86.

\bibitem{BG10} D. Bucur \& A. Giacomini.  A variational approach to the isoperimetric inequality for the Robin eigenvalue problem.
Arch. Rational Mech. Anal. 198 (2010), 927--961.

\bibitem{BG15} D. Bucur \&A. Giacomini.  Faber-Krahn inequalities for the Robin-Laplacian: A free discontinuity approach.
Arch. Rational Mech. Anal. 218 (2015),  757--824.


\bibitem{D1} D. Daners.
A Faber-Krahn inequality for Robin problems in any space dimension.
Math. Ann.  333 (2006), 767--785.

\bibitem{Hu} R.A. Hunt, An extension of the Marcinkiewicz interpolation theorem to Lorentz spaces. Bull. Amer. Math. Soc. 70 (1964), 803--807.

\bibitem{Ka} B. Kawohl. Rearrangements and convexity of level sets in PDE. Lecture Notes in Mathematics, 1150. Springer-Verlag, Berlin, 1985. 

\bibitem{K} S. Kesavan.  Some remarks on a  result of Talenti. Annali Sc. Norm Sup. Pisa. Cl. Sci. (4) 15 (1988), 453--465.

\bibitem{K1} S. Kesavan.  Symmetrization \& applications. Series in Analysis, 3, World Scientific Publishing Co. Pte. Ltd., Hackensack, NJ, 2006.


\bibitem{Ta} G. Talenti. Elliptic equations and rearrangements.  Ann. Scuola
Norm. Sup. Pisa Cl. Sci. (4)  3 (1976),  697--718.

\bibitem{Ta2} G. Talenti. Nonlinear elliptic equations, rearrangements of functions and Orlicz spacesElliptic equations and rearrangements.  Ann. Mat. Pura e Appl.
120 (1979),  159--184.

\bibitem{Ta3} G. Talenti. On the first eigenvalue of the clamped plate.  Ann. Mat. Pura e Appl.
129 (1981),  265--280.
\end{thebibliography}
\end{document}